\documentclass[final,5p,times,twocolumn]{elsarticle}

\usepackage{multicol}
\usepackage{color}
\usepackage{xspace}

\usepackage{amssymb}
\usepackage{amsthm}
\usepackage{latexsym}
\usepackage{amsfonts}
\usepackage{amsmath}
\usepackage{mathtools}
\usepackage{caption}
\usepackage{algpseudocode}
\usepackage[ruled]{algorithm}
\usepackage{url}
\usepackage[mathscr]{euscript}

\newtheorem{proposition}{Proposition}

\newtheorem{theorem}{Theorem}

\newtheorem{remark}{Remark}

\renewcommand{\theta}{\vartheta}

\newcommand{\Sy}{{\mathscr S}_n}
\newcommand{\sdp}{{\mathscr S}_n^{+}}
\newcommand{\sdn}{{\mathscr S}_n^{-}}

\newcommand{\Aat}{{\mathscr A}^{\top}}

\newcommand{\Aa}{{\mathscr A}}

\newcommand{\RR}{\mathbb{R}}

\newcommand{\Diag}{{\rm Diag}}
\newcommand{\diag}{{\rm diag}}

\newcommand{\ignore}[1]{}
\newcommand\R{{\mathbb R}}

\def\ignore#1{}

\journal{OR letters}

\begin{document}

\begin{frontmatter}

\title{Using a Factored Dual in Augmented Lagrangian Methods for Semidefinite Programming}

\author{Marianna De Santis\fnref{diag}}
\ead{marianna.desantis@uniroma1.it}

\author{Franz Rendl\fnref{aau}}
\ead{franz.rendl@aau.at}

\author{Angelika Wiegele\fnref{aau}}
\ead{angelika.wiegele@aau.at}

\address[diag]{Dipartimento di Ingegneria Informatica Automatica e Gestionale, Sapienza Universit\`{a} di Roma, Via Ariosto, 25, 00185 Roma, Italy\fnref{label1}}
\address[aau]{Institut f\"ur Mathematik, Alpen-Adria-Universit\"at Klagenfurt, Universit\"atsstra{\ss}e 65-67, 9020 Klagenfurt, Austria \fnref{label2}\fnref{label3}}

\begin{abstract}
In the context of augmented Lagrangian approaches for solving semidefinite 
programming problems, we investigate the possibility of eliminating the positive semidefinite constraint on the dual matrix by employing a factorization. 
Hints on how to deal with the resulting unconstrained maximization of the augmented Lagrangian are given. We further use the approximate maximum of 
the augmented Lagrangian with the aim of improving the convergence rate of alternating direction augmented Lagrangian frameworks. 
Numerical results are reported, showing the benefits of the approach.\end{abstract}


\end{frontmatter}
\section{Introduction}\label{sec:intro}
Semidefinite Programs (SDP) can be solved in polynomial time to some fixed prescribed precision, but the computational 
effort grows both with the number
$m$ of constraints and with the order $n$ of the underlying space of symmetric matrices. 
Interior point methods to solve SDP become impractical both in terms of computation time and memory requirements, 
once $m \geq 10^4$. 
Several algorithmic alternatives have been introduced in the literature, including some based on 
augmented Lagrangian approaches \cite{BuMo:03,BuMo:05,PoReWi:06,MaPoReWi:09,Wen2010}.
It is the purpose of this paper to elaborate on the 
alternating direction augmented Lagrangian (ADAL) algorithms proposed in
\cite{PoReWi:06,MaPoReWi:09} by introducing computational refinements.
The key idea will be to 
eliminate the positive semidefinite constraint on the dual matrix
by employing a factorization, so that the maximization of the 
augmented Lagrangian function with respect to the dual variables 
can be performed in an unconstrained fashion.

In the remainder of this section we give
the problem formulation and state our notations.
In Section~\ref{sec:AugmLag} a description
of the ADAL methods for solving
semidefinite programs is given. 
Details on how we maximize the augmented Lagrangian, 
after the factorization of the dual matrix, are given in Section~\ref{sec:SolveSub}.
In Section~\ref{dadal}, we outline 
our new algorithm~\texttt{DADAL}: an additional update of the dual variable
within one iteration of the ADAL method is used as improvement step.
The convergence of~\texttt{DADAL} easily follows by the analysis done
in~\cite{Wen2010}, that looks at ADAL as a fixed point method. We give
insights on how~\texttt{DADAL} can improve the convergence rate of ADAL in
Section~\ref{sec:conv}.
Section~\ref{sec:num} shows numerical results and Section~\ref{sec:concl} concludes.

\subsection{Problem Formulation and Notations}
Let $\Sy$ be the set of $n$-by-$n$ symmetric matrices and $\sdp~\subset~\Sy$ be the set of positive semidefinite matrices.
Denoting by $\left\langle X,Y\right\rangle = \mathrm{trace}(XY)$ the standard inner product in $\Sy$,
we write the standard primal-dual pair of SDP problems as follows:
\begin{equation}
\begin{array}{l l}
\min & \left\langle C,X\right\rangle\\
\mbox{ s.t. } & \Aa X = b, \\
 & X\in \sdp
\end{array}
\label{eq:primal}
\end{equation}
and
\begin{equation}
\begin{array}{l l}
\max  & b^Ty\\
\mbox{ s.t. } & C - \Aat y = Z \\
& Z\in \sdp,
\end{array}
\label{eq:dual}
\end{equation}
where $C\in\Sy$, $b\in\RR^m$, $\Aa: \Sy \rightarrow \RR^m$ is the linear operator 
$(\Aa X)_i = \left\langle A_i,X\right\rangle$ with $A_i\in\Sy$, $i=1,\ldots,m$ and 
$\Aat: \RR^m \rightarrow \Sy$ is its adjoint, $\Aat y = \sum_i y_i A_i$.

We assume that both problems have strictly feasible points, so that strong duality holds.
Under this assumption, $(X, y, Z)$ is optimal if and only if
\begin{equation}\label{eq:optCond}
 X\in \sdp, \quad \Aa X = b, \quad Z\in \sdp, \quad C - \Aat y = Z, \quad ZX = 0.
\end{equation}
We further assume that matrix $A$ has full rank.

Let $v\in \RR^n$ and $M\in \RR^{m\times n}$. In the following, we denote by
$\mathrm{vec}(M)$ the $mn$-dimensional vector formed by stacking the columns of $M$ on top of each other 
($\mathrm{vec}^{-1}$ is the inverse operation). We also denote by $\Diag{(v)}$ the diagonal matrix having 
$v$ in the diagonal. With $e_i$ we denote the $i$-th vector of the standard basis in $\RR^n$. 
Whenever a norm is used, we consider the Frobenius norm in case of matrices and the Euclidean norm 
in case of vectors. We denote the projection of some symmetric matrix $S$ onto
the positive semidefinite cone by $(S)_{+}$ and its projection onto the
negative semidefinite cone by $(S)_{-}$.

\section{Augmented Lagrangian Methods for SDP}\label{sec:AugmLag}
Let $X\in \Sy$ be the Lagrange multiplier for the dual equation 
$Z - C + \Aat y=0$.
In order to solve Problem~\eqref{eq:dual} with 
the augmented Lagrangian method we introduce
$$
L_\sigma(y,Z; X) := b^Ty - \langle Z - C + \Aat y, X\rangle -
\frac{\sigma}{2}\|Z - C + \Aat y\|^2.$$
To solve \eqref{eq:dual}, we deal with
\begin{equation}\label{eq:maxLagZ}
 \begin{array}{l l}
      \max  & L_\sigma(y,Z; X) \\
      \mbox{ s.t. } & y\in \RR^m, \quad Z\in \sdp,
\end{array}
\end{equation}
where $X$ is fixed and $\sigma>0$ is the penalty parameter.

Once Problem~\eqref{eq:maxLagZ} is (approximately) solved, the multiplier $X$ is updated by a first order rule:
\begin{equation}\label{eq:firstOrderUpdate}
 X = X + \sigma (Z -C +\Aat y)
\end{equation}
and the process is iterated until convergence, i.e. until the optimality conditions~\eqref{eq:optCond} are satisfied
within a certain tolerance (see Chapter 2 in \cite{Be:82} for further details).

Problem~\eqref{eq:maxLagZ} is a convex quadratic semidefinite optimization problem,
which is tractable but expensive to solve directly. Several simplified
versions have been proposed in the literature to quickly get approximate
solutions.

In the alternating direction framework proposed by Wen et al.\ \cite{Wen2010}, 
the augmented Lagrangian $L_\sigma(y,Z; X)$ is maximized
 with respect to $y$ and $Z$ one after the other. 
More precisely, at every iteration $k$, 
the new point $(X^{k+1}, y^{k+1}, Z^{k+1})$ is computed by the following steps:
\begin{equation} \label{eq:opty}
 y^{k+1} := \arg \max_{y\in \RR^m} L_{\sigma^k}(y,Z^{k}; X^{k}),
\end{equation}
\begin{equation} \label{eq:optZ}
 Z^{k+1} := \arg\max_{Z\in \sdp} L_{\sigma^k}(y^{k+1},Z; X^{k}),
\end{equation}
\begin{equation} \label{eq:updateX}
 X^{k+1} := X^{k} + \sigma^k (Z^{k+1} -C + \Aat y^{k+1}).
\end{equation}
These three steps are iterated until a stopping criterion is met.

The update of $y$ in~\eqref{eq:opty} is derived from the first-order optimality 
condition of Problem~\eqref{eq:maxLagZ}: $y^{k+1}$ 
is the unique solution of
\[
\nabla_y L_{\sigma^k} (y,Z^{k}; X^{k}) = b - \Aa( X^k + \sigma^k (Z^k -C + \Aat y)) = 0,
\]
that is
\[
y^{k+1}= (\Aa\Aat)^{-1}\Big(\frac{1}{\sigma^k} b - \Aa(\frac{1}{\sigma^k} X^k -C + Z^k)\Big).
\]
Then, the maximization in~\eqref{eq:optZ} is conducted by considering the equivalent problem
\begin{equation}\label{eq:proj}
 \min_{Z\in \sdp} \|Z + W^k\|^2,
\end{equation}
with $W^k = (\frac{X^k}{\sigma^k} - C + \Aat y^{k+1})$, or, in other words, by projecting $W^k\in \Sy$ 
onto the (closed convex) cone $\sdn$ and taking its additive inverse (see Algorithm~\ref{alg:BPM}). 
Such a projection is computed via the spectral decomposition of the matrix
$W^k$.

The Boundary Point Method proposed in~\cite{PoReWi:06, MaPoReWi:09} 
can also be viewed as an 
alternating direction augmented Lagrangian method. 
In fact, as also noted in~\cite{Wen2010}, the available implementation \texttt{mprw.m}
(see \url{https://www.math.aau.at/or/Software/}) is an alternating 
direction augmented Lagrangian method,
since in the inner loop (see Table~2 in \cite{PoReWi:06}) 
only one iteration is performed.

We report in Algorithm~\ref{alg:BPM}, the scheme of \texttt{mprw.m}.
The stopping criterion for \texttt{mprw.m} considers only the following primal and dual infeasibility errors:
\[
r_P= \frac{\|\Aa X - b\|}{1+\|b\|},
\]
\[
r_D= \frac{\|C - Z - \Aat y\|}{1+ \|C\|};
\]
as the other optimality conditions (namely, $X\in \sdp, Z\in \sdp, ZX=0$) 
are satisfied up to machine accuracy throughout the algorithm.
More precisely,
the algorithm stops as soon as the quantity
\[
\delta = \max \{ r_P, r_D\},
\]
is less than a fixed precision $\varepsilon>0$.
\begin{algorithm}                      
\caption{\small\tt Scheme of mprw.m}          
\label{alg:BPM} 
\begin{algorithmic} 
\par\vspace*{0.1cm}
\item$\,\,\,1$\hspace*{0.5truecm} {\bf Initialization:} Choose $\sigma>0$, $X\in \sdp$, $\varepsilon >0$. 
\item$\,\,\,$ \hspace*{2.75truecm}  Set $Z= 0$.
\item$\,\,\,2$\hspace*{0.5truecm} {\bf Repeat until } $\delta < \varepsilon$:
\item$\,\,\,3$\hspace*{1.0truecm} {\bf Compute} $y= (\Aa\Aat)^{-1}\Big(\frac{1}{\sigma} b - \Aa(\frac 1 \sigma Y -C + Z)\Big)$ 
\item$\,\,\,4$\hspace*{1.0truecm} {\bf Compute} $Z =-(X/\sigma -C + \Aat y)_-$ 
\item$\,\,\, $\hspace*{2.65truecm}  $X =\sigma(X/\sigma -C + \Aat y)_+$
\item$\,\,\,5$\hspace*{1.0truecm} {\bf Compute} $\delta = \max \{ r_P, r_D\}$
\item$\,\,\,6$\hspace*{1.0truecm} {\bf Update} $\sigma$ 
\par\vspace*{0.1cm}
\end{algorithmic}
\end{algorithm}

It is the main purpose of this paper to investigate enhancements to this
algorithm. The key idea will be to replace in the subproblem  
\eqref{eq:maxLagZ} the constraint $Z \in \sdp$ by $Z=VV^\top$ and considering
\eqref{eq:maxLagZ} as an unconstrained problem in $y$ and $V$.

\section{Solving the subproblem \eqref{eq:maxLagZ}}\label{sec:SolveSub}
By introducing a variable $V\in \RR^{n \times r}$ ($1 \leq r\leq n$), 
such that $Z = VV^\top \in \sdp$,
we reformulate Problem~\eqref{eq:maxLagZ} as the following unconstrained maximization problem
\begin{equation}\label{eq:maxLagV}
 \begin{array}{l l}
\max  & L_\sigma (y,V; X)\\
 \mbox{ s.t. } & y\in \RR^m, \quad V\in  \RR^{n \times r},
\end{array}
\end{equation}
where 
\[L_\sigma (y,V; X) = b^Ty - \langle VV^\top - C + \Aat y, X\rangle - 
\frac{\sigma}{2}\|VV^\top - C + \Aat y\|^2.\]
Note that the number of columns $r$ of matrix $V$ represents the rank of the dual variable $Z$.

The first-order necessary optimality conditions for Problem~\eqref{eq:maxLagV} state the following:
\begin{proposition}\label{prop:NecOptCond}
Let $(y^*, V^*)\in \RR^m \times \RR^{n \times r}$ be a stationary point for Problem~\eqref{eq:maxLagV}, then
\begin{equation}\label{eq:OptCondLag}
 \begin{array}{l l}
  \nabla_y L_\sigma (y^*,V^*; X) = b - \Aa( X + \sigma ({V^*V^*}^\top -C + \Aat y^*)) = 0,\\[1.2ex]
  \nabla_V L_\sigma (y^*,V^*; X) = -2 (X + \sigma ({V^*V^*}^\top -C + \Aat y^*))V^* = 0.
 \end{array}
\end{equation} 
\end{proposition}

From~\eqref{eq:OptCondLag}, we can easily see that we can keep the optimality 
conditions with respect to $y$ satisfied, 
while moving $V$ along any direction $D_V\in  \RR^{n \times r}$:
\begin{proposition}\label{prop:yKeptOpt}
 Let $D_V\in  \RR^{n \times r}$. 
 Let 
\begin{equation}\label{eq:yKeptOpt}
y(V+\alpha D_V) = y_0 + \alpha y_1 + \alpha^2 y_2,
\end{equation} 
with
\[
 \begin{array}{l l}
  y_0 = (\Aa\Aat)^{-1}\Big(\frac{1}{\sigma} b - \Aa(\frac 1 \sigma X -C + VV^\top)\Big),\\[1.2ex]
  y_1 = (\Aa\Aat)^{-1}\Big(-\Aa(D_V V^\top + V D_V^\top)\Big),\\[1.2ex]
  y_2 = (\Aa\Aat)^{-1}\Big(-\Aa (D_V D_V^\top)\Big).
 \end{array}
\]
Then 
\[
\nabla_y L_\sigma (y(V+\alpha D_V),V + \alpha D_V; X) = 0, 
\]
for all $\alpha\in \RR$.
\end{proposition}
\begin{proof}
Let $\alpha\in \RR$. From~\eqref{eq:OptCondLag}, we have that $\nabla_y L_\sigma (y,V; X) = 0$
 iff 
 \[
 \Aa\Aat y= \Big(\frac{1}{\sigma} b - \Aa(\frac 1 \sigma X -C + {VV}^\top)\Big).
 \]
 Therefore, when  $V = V + \alpha D_V$, we get
 \begin{align*}
  (\Aa\Aat) y =& \, \frac{1}{\sigma} b - \Aa\Big(\frac{1}{\sigma} X - 
  C + (V+\alpha D_V)(V + \alpha D_V)^\top \Big)\\[1.2ex]
  =& \, \frac{1}{\sigma} b - \Aa\Big(\frac{1}{\sigma} X - C + VV^\top\Big) \\[1.2ex]
   & \,- \alpha \Aa(D_V V^\top + V D_V^\top) -\alpha^2 \Aa(D_V D_V^\top).
  \end{align*}
 By multiplying both the l.h.s. and the r.h.s. with $(\Aa\Aat)^{-1}$, we get the expression in~\eqref{eq:yKeptOpt} and the proposition is proven.
\end{proof}
Thanks to Proposition~\ref{prop:yKeptOpt}, we can maximize 
$L_\sigma (y,V; X)$ with respect to $V$, 
keeping variable $y$ updated according to~\eqref{eq:yKeptOpt} along the 
iterations. Thus we are in fact maximizing a polynomial of degree 4 in
$V$.

\subsection{Direction Computation}\label{sec:ScaledDir}
In order to compute an ascent direction for the augmented Lagrangian we consider two possibilities.
Either we use the gradient of $L_\sigma(y,V; X)$ with respect to $V$, or we use the gradient scaled 
with the inverse of the diagonal of the  Hessian of $L_\sigma (y,V; X)$. 
We recall the gradient of $L_\sigma$ with respect to $V$, see 
\eqref{eq:OptCondLag}, 
 as the $n \times r$ matrix 
\[  
\nabla_V L_\sigma (y,V; X) = - 2(M + \sigma VV^\top)V,\]
where $M=X + \sigma(\Aat y -C )$.
In order to compute the generic $(s,t)$ entry on the main 
diagonal of the Hessian, we consider 
\[\lim_{t\rightarrow 0} \,\frac{1}{t}\, \left(f_{st}(V + t e_s e_t^\top) - f_{st}(V)\right), \]
where $f_{st}(V) := e_s ^\top VV^\top V e_t,\;\; s\in \{1,\ldots, n\}, \;t\in \{1, \ldots, r\}$.  

We get
\[
\frac{\partial^{2}L_{\sigma}(y,V; X)}{\partial v_{s,t} \partial
  v_{s,t}} = -2m_{ss} -2\sigma\left( (e_{s}^\top Ve_{t})^2 + \|V^\top e_{s} \|^{2} +\|Ve_{t}\|^{2} \right). 
\]

We define the $n \times r$ matrix $H$ by
$$
   (H)_{s,t} = 2\max \{ 0, m_{ss} \}\, +  
   2\sigma\, \big((e_s^\top V e_t)^2 + \|V^\top e_s\|^2 +\|V e_t\|^2\big),
$$
and we propose to use the search direction $D_{V}$, given by the gradient 
scaled by $H$, thus
\begin{equation}\label{eq:dir}
(D_{V})_{s,t} := \frac{(\nabla_{V}L_{\sigma})_{s,t}}{H_{s,t}}.
\end{equation}
We note that $H$ is generically positive, as $V$ should not contain columns
all equal to zero. In practice, 
we compute the scaled gradient direction in case $\|\nabla_V L_\sigma (y,V; X)\| < 10^{-3}$.
\subsection{Exact Linesearch}
Given a search direction $D_{V}$, we note that 
$L_\sigma (y(V + \alpha D_{V}),V + \alpha D_{V}; X)$ is a polynomial of degree 4 
in $\alpha$, 
so that we can interpolate five different points 
\[(\alpha_i, \; 
L_\sigma (y(V + \alpha_i D_V),V + \alpha_i D_V; X)\,), \quad i= 1,\ldots, 5,\]
to get its analytical expression. 
This also means that the maximum of $L_\sigma (y(V + \alpha D_V),V + \alpha
D_V; X)$ can be detected analytically
(using the Cardano formula).
In practice, we evaluate the $4$-degree polynomial 
$L_\sigma (y(V + \alpha D_V),V + \alpha D_V; X)$ in few thousands of points in $(0,10)$ and take the best $\alpha$.

Algorithm~\ref{alg:OneItUncMax} is a scheme of the generic iteration we
perform to maximize $L_\sigma(y,V; X)$, $X$ being fixed.
\begin{algorithm}                      
\caption{\small\tt Solving \eqref{eq:maxLagZ} approximately}          
\label{alg:OneItUncMax} 
\begin{algorithmic} 
\par\vspace*{0.1cm}
\item$\,\,\,1$\hspace*{0.5truecm} {\bf Input:}  $\sigma>0$, $y\in \RR^m$, 
  $V\in \RR^{n\times r}$ 
\item$\,\,\,2$\hspace*{0.5truecm} {\bf Repeat until} $\| \nabla_{V}
  L_{\sigma}\|<\epsilon_{inner}$ 
\item$\,\,\,3$\hspace*{0.5truecm} {\bf Compute} the search direction $D_V\in
  \RR^{n\times r}$ using \eqref{eq:dir}
\item$\,\,\,4$\hspace*{0.5truecm} {\bf Compute} optimal stepsize $\alpha$ 
\item$\,\,\,5$\hspace*{0.5truecm} {\bf Update} $y = y(V + \alpha D_V)$ according to~\eqref{eq:yKeptOpt}\\
\hspace*{2.1truecm}and $V = V + \alpha D_V$.
\par\vspace*{0.1cm}
\end{algorithmic}
\end{algorithm}

\section{DADAL: a dual step for 
improving alternating direction augmented Lagrangian methods for SDP }\label{dadal}
Our idea is to insert the approximated solution of Problem~\eqref{eq:maxLagZ} obtained from Algorithm~\ref{alg:OneItUncMax},
as a simultaneous update of $y$ and $V$ to be performed before the projection step in Algorithm~\ref{alg:BPM}.
We detail below the scheme of the ADAL method 
where this ``dual step'' is inserted.
We refer to Algorithm~\ref{alg:BPMComb} as \texttt{DADAL}. 
\begin{algorithm}                      
\caption{\small\tt DADAL}          
\label{alg:BPMComb} 
\begin{algorithmic} 
\par\vspace*{0.1cm}
\item$\,\,\,1$\hspace*{0.5truecm} {\bf Initialization:} Choose $\sigma>0$, $r>0$, $\varepsilon >0$,
\item$\,\,\,$\hspace*{3.0truecm} $X\in \sdp$, $V\in \R^{n\times r}$, $Z=VV^{\top}$, $y\in \R^m$.
\item$\,\,\,2$\hspace*{0.5truecm} {\bf Repeat until } $\delta < \varepsilon$:
\item$\,\,\,3$\hspace*{1.0truecm} {\bf Update} $(y,V)$ by Algorithm~\ref{alg:OneItUncMax}
\item$\,\,\,4$\hspace*{1.0truecm} {\bf Compute} 
$Z = - (X/\sigma -C +\Aat y)_- $
\item$\,\,\,$\hspace*{2.6truecm}  $X =\sigma(X/\sigma - C + \Aat y)_+$
\item$\,\,\,5$\hspace*{1.0truecm} {\bf Update} $r\leftarrow \mathrm{rank(Z)}$ and $V$ so that $VV^\top = Z$
\item$\,\,\,6$\hspace*{1.0truecm} {\bf Compute} $\delta = \max\{ r_P, r_D\}$
\item$\,\,\,7$\hspace*{1.0truecm} {\bf Update} $\sigma$ 
\par\vspace*{0.1cm}
\end{algorithmic}
\end{algorithm}

\subsection{Convergence Analysis}\label{sec:conv}
The convergence analysis of \texttt{DADAL} may follow from the analysis in~\cite{Wen2010}, 
that looks at ADAL as a fixed point method.
One iteration of the ADAL method can be seen as the result of the combination of two operators,
namely $(Z^{k+1},X^{k+1}) = \mathscr{P}(W(Z^{k}, X^{k}))$, where $\mathscr{P}$ denotes the projection performed
at Step 4 in Algorithm~\ref{alg:BPM}
and $W(Z^{k}, X^{k}) = X^k/\sigma - C + \Aat y(Z^{k}, X^{k})$, being   
$y(Z^{k}, X^{k})= (\Aa\Aat)^{-1}\Big(b/\sigma^k - \Aa(X^k/\sigma^k -C + Z^k)\Big)$.
It can be shown that $\mathscr{P}$ and $W$ are non-expansive operators (see~\cite{Wen2010} for further details). 
Hence, the key step in the proof of Theorem 2 in~\cite{Wen2010} states the following
\begin{align} \label{eq:proofWGY}
 \|(Z^{k+1},X^{k+1}/\sigma) - (Z^{*}, X^{*}/\sigma)\| & = 
 \|\mathscr{P}(W(Z^{k}, X^{k})) - \mathscr{P}(W(Z^{*}, X^{*}))\| \nonumber \\ 
 & \le \| W( Z^{k}, X^{k}) - W(Z^{*}, X^{*}) \| \\
 & \le \| (Z^{k}, X^{k}/\sigma) - (Z^{*}, X^{*}/\sigma) \|.\nonumber
\end{align}
In \texttt{DADAL}, before the projection step, the dual variables are updated 
by performing one maximization step for the augmented Lagrangian, so that we get 
$(\hat y^k, \hat V^k)$ from $(y^k, V^k)$ and, in particular, the dual matrix is given as 
$\hat Z^k = \hat V^k \hat{V^k}^\top$. 
\begin{proposition}\label{prop:deltabar}
 Let $\hat Z, Z, X, X^*, Z^* \in \Sy$ and let $Z\neq Z^*$.
 Let \[\|\hat Z - Z^*\|^2 \leq \rho \|Z - Z^*\|^2,\] with $0<\rho\leq 1$.
 Then $\bar\rho$ exists, $\bar\rho \leq 1$, $\bar\rho \geq \rho$ such that
 \[\| (\hat Z, X) - (Z^{*}, X^{*}) \|^2 \leq \bar\rho \| (Z, X) - (Z^{*}, X^{*}) \|^2\]
\end{proposition}
\begin{proof}
\begin{align} \label{eq:proof-deltabar}
\| (\hat Z, X) - (Z^{*}, X^{*}) \|^2 & = \|\hat Z - Z^{*}\|^2 + \|X - X^{*}\|^2 \nonumber\\
& \leq \rho \|Z - Z^{*}\|^2 + \|X - X^{*}\|^2 \nonumber\\
& = \bar\rho \| (Z, X) - (Z^{*}, X^{*}) \nonumber\|^2,
\end{align}
where $\bar\rho = \rho + \varepsilon$, with 
\[
\varepsilon = (1 - \rho)\, c(X,Z),
\]
and  \[ c(X,Z) := \frac{\|X - X^{*}\|^2}{\| (Z, X) - (Z^{*}, X^{*}) \|^2}.\]
Since $0\leq c(X,Z)<1$, we have $0\leq \varepsilon < (1-\rho)$, so that $\bar\rho \leq 1$ and $\bar\rho \geq \rho$.
\end{proof}
Note that, if $\rho <1$, we have that $\bar\rho = \rho + \varepsilon < 1$.
\begin{theorem}
Let $(\hat y^k, \hat Z^k)$ be the dual variables obtained from $(y^k, Z^k)$ by performing one iteration 
of Algorithm~\ref{alg:OneItUncMax}. Let the direction $D_V$ in 
Algorithm~\ref{alg:OneItUncMax}, be chosen such that 
\begin{equation}\label{CondDv}
\| \hat Z^k - Z^*\|\le \delta \| Z^k - Z^*\|, 
\end{equation}
with $\delta \le 1$. Then the sequence $\{X^k, y^k, Z^k\}$ generated by Algorithm~\ref{alg:BPMComb} converges 
to a solution $\{X^*, y^*, Z^*\}$.
\end{theorem}
\begin{proof}
Since condition~\eqref{CondDv} holds, we can apply Proposition~\ref{prop:deltabar}, so that $\bar \delta$ exists, $\delta \leq \bar\delta\leq 1$, such that 
\[
\| (\hat Z^{k}, X^{k}/\sigma) - (Z^{*}, X^{*}/\sigma) \| \le \bar \delta \| (Z^{k}, X^{k}/\sigma) - (Z^{*}, X^{*}/\sigma) \|.
\]
Then, the series of inequalities~\eqref{eq:proofWGY} 
can be extended as:
\begin{align*}
 \|(Z^{k+1},X^{k+1}/\sigma) - (Z^{*}, X^{*}/\sigma)\| & = \|\mathscr{P}(W(\hat Z^{k}, X^{k})) - 
 \mathscr{P}(W(Z^{*}, X^{*}))\| \\
 & \le \| W(\hat Z^{k}, X^{k}) - W(Z^{*}, X^{*}) \| \\
 & \le \| (\hat Z^{k}, X^{k}/\sigma) - (Z^{*}, X^{*}/\sigma) \| \\
 & \le \bar \delta \| (Z^{k}, X^{k}/\sigma) - (Z^{*}, X^{*}/\sigma) \|.
\end{align*}
The rest of the proof follows the same arguments as those in Theorem 2 in~\cite{Wen2010}.
\end{proof}
\par\noindent
Note that $\delta < 1$ in~\eqref{CondDv} implies $\bar \delta <1$ and the additional step of maximizing the augmented Lagrangian is
strictly improving the convergence rate of ADAL methods.
\begin{remark}
When dealing with unconstrained optimization problems, it is well known that the sequence $\{x^k\}$ produced by Newton's method 
converges superlinearly to a stationary point $x^*$ if the starting 
point $x^0$ is sufficiently close to $x^*$ (see e.g. Prop 1.4.1 in \cite{Be:99}).
Therefore, condition~\eqref{CondDv} is satisfied when, e.g., $D_V$ is chosen as the Newton direction and
our starting $V$ in Algorithm~\ref{alg:OneItUncMax} is in a neighborhood of the optimal solution. 
Assumption~\eqref{CondDv} is in fact motivating our direction computation: as soon as we are close enough to an optimal solution, 
we try to mimic the Newton direction by scaling the gradient with the inverse of the diagonal of the Hessian~\eqref{eq:dir}.
\end{remark}

\subsection{Choice and Update of the Penalty Parameter}\label{sec:sigma}
Let $(y^0, Z^0, X^0)$ be our starting solution.
In defining a starting penalty parameter $\sigma^0$, 
i.e. a starting value for scaling the violation of the dual equality constraints in the 
augmented Lagrangian, we might want to take into account the dual infeasibility error $r_D$ 
(defined in Section~\ref{sec:AugmLag}) 
at the starting solution.
Since the penalty parameter enters in the update of the primal solution~\eqref{eq:updateX} as well,
we can tune $\sigma^0$ so that the starting primal and dual infeasibility errors
are balanced.

Our proposal is to use the following as starting penalty parameter:
\begin{equation}\label{eq:sigma0}
\sigma^0 = \frac{r_P}{r_D}\frac{\|\Aa X^0 - b\|}{\|C - Z^0 - \Aat y^0\|}
\end{equation}
It has been noticed that the update of the penalty parameter $\sigma$ 
is crucial for the computational performances of ADAL methods for SDPs~\cite{PoReWi:06, MaPoReWi:09, Wen2010}.
Therefore, in order to improve the numerical perfomance of \texttt{DADAL}, strategies to dynamically 
adjust the value of $\sigma$ may be considered: In the implementation of \texttt{DADAL} used in our numerical experience (see
Section~\ref{sec:num}), we adopt the following strategy. 

We monitor the primal and dual relative errors, and note that the stopping
condition of the augmented Lagrangian method is given by $r_{D} \leq 
\varepsilon$, provided that the inner problem is solved with reasonable
accuracy. In order to improve the numerical performance, we use the following
intuition. 
If $r_{D}$ is much smaller than $r_{P}$ (we test for $100r_{D}< r_{p}$), this
indicates that $\sigma$ is too big in the subproblem. 
On the other hand, if $r_{D}$ is much larger than $r_{p}$ 
(we test for $2r_{D} > r_{P}$), then $\sigma$ should be increased to
facilitate overall progress. If either of these two conditions
occurs consecutively for several iterations, we change
$\sigma$ dividing or multiplying it by 1.3. Similar heuristics have been suggested also in 
\cite{Wen2010} and \cite{MaPoReWi:09}.


\section{Numerical Results}\label{sec:num}
In this section we report our numerical experience:
we compare the performance of \texttt{DADAL} and \texttt{mprw.m} on randomly 
generated instances, on instances from 
the SDP problem underlying the Lov\'asz theta number of a graph 
and on linear ordering problem instances.
Both \texttt{mprw.m} and \texttt{DADAL} are implemented in MATLAB~R2014b and are available at~\url{https://www.math.aau.at/or/Software/}. 
In our implementation of Algorithm~\ref{alg:BPMComb}, we use the choice and update strategy for the penalty
parameter $\sigma$ described in Section~\ref{sec:sigma}
and we perform two iterations of Algorithm~\ref{alg:OneItUncMax} in order to update $(y,V)$ in Step~3.
We set the accuracy~level~$\varepsilon~=~10^{-5}$.

We also report the run time of the interior point method for SDP using the solver \texttt{MOSEK}~\cite{mosek}.
The experiments were carried out on an Intel Core i7 processor running at
3.1~GHz under Linux.
\subsection{Comparison on randomly generated instances}
The random instances considered in the first experiment are some of those used in~\cite{MaPoReWi:09} 
(see Table~\ref{tab:randinst}). 
For these instances, the Cholesky factor of $\Aa \Aat$ is computed once and then used along the iterations
in order to update the dual variable $y$.

%
As can be seen in Table~\ref{tab:resRand}, interior point methods are not able to solve instances with more than $30000$ 
constraints due to memory limitations.
For what concerns \texttt{DADAL}, the number of iterations decreases.
However, we observe that a decrease in the
number of iterations does not always correspond to an improvement in the computational 
time: This suggests that the update of 
the dual variables performed at Step~3 in~\texttt{DADAL} may be too expensive.
According to the MATLAB profiling of our code, this is due to the need of solving three
linear systems for the update of $y$ (see Proposition~\ref{prop:yKeptOpt}).
\begin{table}
\small{
\centering      
  \begin{tabular}{|c||r|r|r|r|} 
    \hline
    \hline    
   Problem &  $n$   &  $m$   &  p & seed \\
  \hline
  \hline      
  P1 & 300 &   20000 & 3 &    3002030 \\    
  P2 & 300 &   25000 & 3 &    3002530 \\
  P3 & 300 &   10000 & 4 &    3001040 \\
  \hline
  P4 & 400 &   30000 & 3 &    4003030 \\   
  P5 & 400 &   40000 & 3 &    4004030 \\  
  P6 & 400 &   15000 & 4 &    4001540 \\ 
  \hline
  P7 & 500 &   30000 & 3 &    5003030 \\  
  P8 & 500 &   40000 & 3 &    5004030 \\  
  P9 & 500 &   50000 & 3 &    5005030 \\  
  P10 & 500 &   20000 & 4 &    5002040 \\
  \hline
  P11 & 600 &   40000 & 3 &    6004030 \\
  P12 & 600 &   50000 & 3 &    6005030 \\
  P13 & 600 &   60000 & 3 &   6006030 \\
  P14 & 600 &   20000 & 4 &    6002040 \\
  \hline
  P15 & 700 &   50000 & 3 &    7005030 \\  
  P16 & 700 &   70000 & 3 &    7007030 \\
  P17 & 700 &   90000 & 3 &    7009030 \\
  \hline
  P18 & 800 &   70000 & 3 &    8007030 \\  
  P19 & 800 &  100000 & 3 &   80010030 \\
  P20 & 800 &  110000 & 3 &   80011030 \\
   \hline
  \hline
 \end{tabular}\vspace*{0.1cm}
 \caption{Randomly generated instances.}
 \label{tab:randinst}  
 }
\end{table}
 \begin{table}
 \small{
 \centering      
 \begin{tabular}{|c|r||r|r||r|r|} 
   \hline
   \hline
   \multicolumn{1}{|c|}{} & \multicolumn{1}{|c||}{\texttt{MOSEK}}
   & \multicolumn{2}{c||}{\texttt{mprw.m}} & 
    \multicolumn{2}{c|}{\texttt{DADAL}}\\
   
  Problem & time(s) &  iter & time(s) & iter & time(s) \\
  \hline
  \hline
  P1   & 633.1 &   340 &    \textbf{10.1} &     \textbf{68} &    10.2\\ 
  P2   & 2440.4 &   427 &    33.1 &    \textbf{76} &    \textbf{31.2}\\ 
  P3   &  129.1 &   260 &    \textbf{11.3} &  \textbf{146} &    32.6\\ 
  \hline
  P4   & 7514.6 &   306 &    \textbf{13.3} &  \textbf{101} &    18.9\\
  P5   & - &   376 &    \textbf{51.4} &  \textbf{73} &    56.3\\ 
  P6   & 375.7 &   255 &    \textbf{19.5} &  \textbf{187} &    72.9\\
  \hline
  P7   & 7388.3 &   268 &    \textbf{11.5} &  \textbf{151} &    23.1\\ 
  P8   & - &   289 &    \textbf{15.2} &  \textbf{130} &    30.7\\  
  P9   & - &   319 &    \textbf{35.3} &  \textbf{111} &    74.4\\  
  P10  & 886.0 &   251 &    \textbf{26.9} &  \textbf{222} &    119.8\\ 
  \hline
  P11  & - &   266 &    \textbf{17.3} & \textbf{177} &    41.9\\  
  P12  & - &   275 &    \textbf{18.6} & \textbf{148} &    42.1\\  
  P13  & - &   293 &    \textbf{28.4} &  \textbf{132} &    68.7\\
  P14  & 1156.3 &   \textbf{249} &    \textbf{20.3} &  270 &    116.2\\  
  \hline
  P15  & - &   262 &    \textbf{22.4} &  \textbf{207} &    83.7\\ 
  P16  & - &   278 &    \textbf{30.1} &  \textbf{151} &    79.2\\ 
  P17  & - &   303 &    \textbf{74.2} &  \textbf{128} &   181.5\\  
  \hline
  P18  & - &   264 &    \textbf{34.3} &  \textbf{207} &  111.2\\ 
  P19  & - &   285 &    \textbf{53.3} & \textbf{157} &   126.4\\
  P20  & - &   296 &    \textbf{80.9} & \textbf{133} &   168.6\\
     \hline
     \hline
 \end{tabular}\vspace*{0.1cm}
 \caption{Comparison on randomly generated instances.}
 \label{tab:resRand}
 }
\end{table}

\subsection{Computation of the Lov\'asz theta number}
Given a graph $G$, let $V(G)$ and $E(G)$ be its set of vertices and its set of edges, 
respectively. The Lov\'asz theta number $\theta(G)$ of $G$ is defined as the optimal value of the following 
SDP problem:
\begin{align*}\label{theta}
 \max & \,\left\langle J, X \right\rangle \nonumber \\
 \mbox{ s.t. } & X_{ij} = 0, \quad \forall \, ij \in E(G),\\
  & \mathrm{trace X} = 1, \quad  X\in \sdp,\nonumber 
\end{align*}
where $J$ is the matrix of all ones.

In Table~\ref{tab:inst}, we report the dimension and the optimal value of the theta instances considered, 
obtained from some random graphs of the Kim-Chuan Toh collection~\cite{To:03}. 
As in the case of randomly generated instances, \texttt{MOSEK} can only solve instances with $m< 25000$.

In the case of theta instances $\Aa \Aat$ is a diagonal matrix, so that the update of 
the dual variable $y$ turned out to be less expensive with respect to the case of random instances.
We can see in Table~\ref{tab:theta} the benefits of using \texttt{DADAL}:
the improvements both in terms of number of iterations and in terms of computational time are evident.
We want to underline that for specific instances different tuning of the parameters in~\texttt{DADAL} can further improve the running time.
In Table~\ref{tab:thetaSigmaBest}, we report the results we obtained with~\texttt{DADAL} keeping $\sigma$ fixed to $\sigma^0$ along the
iterations and performing only one iteration of Algorithm~\ref{alg:OneItUncMax} in order to update $(y,V)$ in Step~3: this turned out to be a better choice for 
these instances and resulted in even a better perfomance compared to~\texttt{mprw.m}.  
\begin{table}
\small{
\centering      
  \begin{tabular}{|c||r|r|r|} 
   \hline
   \hline    
   Problem &  $n$   &  $m$   &  $\theta(G)$ \\
  \hline
  \hline
    $\theta$-62 & 300 &   13389 &  29.6413 \\
    $\theta$-82 & 400 &   23871 &  34.3669 \\
    $\theta$-102 & 500 &   37466 & 38.3906 \\
    $\theta$-103 & 500 &   62515 &  22.5286 \\
    $\theta$-104 & 500 &   87244 &  13.3363 \\
    $\theta$-123 & 600 &   90019 &  24.6687 \\
    $\theta$-162 & 800 &   127599 &  37.0097 \\
   $\theta$-1000 & 1000 &  249750 & 31.8053 \\
   $\theta$-1500 & 1500 &  562125 &  38.8665 \\
   $\theta$-2000 & 2000 &  999500 &  44.8558 \\ 
  \hline
  \hline
 \end{tabular}\vspace*{0.1cm}
 \caption{Instances from the Kim-Chuan Toh collection~\cite{To:03}.}
 \label{tab:inst}  
 }
\end{table}
\begin{table}
\small{
\centering      
  \begin{tabular}{|c||r||r|r||r|r|} 
   \hline
   \hline  
   \multicolumn{1}{|c||}{}
    & \multicolumn{1}{c||}{\texttt{MOSEK}} 
    & \multicolumn{2}{c||}{\texttt{mprw.m}} & 
   \multicolumn{2}{c|}{\texttt{DADAL}}\\
   Problem &   time(s) & iter & time(s) & iter & time(s)\\
  \hline
  \hline
    $\theta$-62 &  205.1 &   790 &     9.4 &    \textbf{173} &     \textbf{6.8}\\
    $\theta$-82 &  1174.5 &   821 &    18.2 &    \textbf{143} &     \textbf{11.1}\\
    $\theta$-102 &  - & 852 &    32.2 &    \textbf{163} &     \textbf{20.7}\\
    $\theta$-103 &  - & 848 &    33.8 &    \textbf{192} &     \textbf{25.1}\\
    $\theta$-104 &  - & 873 &    34.0 &    \textbf{290} &     \textbf{32.9}\\
    $\theta$-123 &  - & 874 &    55.0 &    \textbf{191} &     \textbf{38.6}\\
    $\theta$-162 & - &  907 &   108.9 &    \textbf{165} &    \textbf{61.2} \\
   $\theta$-1000 & - & 941 &   212.9 &    \textbf{177} &    \textbf{117.8} \\
   $\theta$-1500 & - &  997 &   803.6 &    \textbf{131} &   \textbf{266.6}\\
   $\theta$-2000 & - &1034 &  1948.5 &    \textbf{111} &   \textbf{478.8}\\ 
  \hline
  \hline
 \end{tabular}\vspace*{0.1cm}
 \caption{Comparison on $\theta$-number instances (from the Kim-Chuan Toh collection).}
 \label{tab:theta}  
 }
\end{table} 
\begin{table}
\small{
 \centering      
 \begin{tabular}{|c||r|r||r|r|} 
   \hline
   \hline
   \multicolumn{1}{|c||}{} 
    & \multicolumn{2}{c||}{\texttt{mprw.m}} & 
   \multicolumn{2}{c|}{\texttt{DADAL}}\\
   Problem &  iter & time(s) & iter & time(s) \\
  \hline
  \hline
   $\theta$-62 &    790 &     9.4 &  \textbf{132} &     \textbf{4.9}\\
   $\theta$-82 &    821 &    18.2 &  \textbf{124} &     \textbf{6.8} \\
   $\theta$-102 &   852 &    32.2 &  \textbf{123} &    \textbf{11.3}\\ 
   $\theta$-103 &   848 &    33.8 &  \textbf{125} &    \textbf{11.7}\\ 
   $\theta$-104 &   873 &    34.0 &  \textbf{166} &    \textbf{15.2}\\ 
   $\theta$-123 &   874 &    55.0 &  \textbf{121} &    \textbf{17.1}\\ 
   $\theta$-162 &   907 &   108.9 &  \textbf{103} &    \textbf{27.4}\\ 
   $\theta$-1000 &   941 &   212.9 &  \textbf{105} &    \textbf{50.6}\\ 
   $\theta$-1500 &   997 &   803.6 &  \textbf{94} &   \textbf{136.8}\\ 
   $\theta$-2000 & 1034 &  1948.5 &   \textbf{89} &   \textbf{283.9}\\ 
  \hline
  \hline
 \end{tabular}\vspace*{0.1cm}
 \caption{Comparison between \texttt{mprw.m} and \texttt{DADAL} on theta number instances - 
  best parameter tuning.}
  \label{tab:thetaSigmaBest}  
  }
\end{table}

\subsection{Comparison on linear ordering problem instances}
Ordering problems associate to each ordering (or permutation) of a set of $n$ objects 
$N = \{1,\ldots,n\}$ a 
profit and the goal is to find an ordering of maximum profit.
In the case of the linear ordering problem (LOP), this profit is determined by
those pairs $(u, v)\in N\times N$, where $u$ comes before $v$ in the ordering.
The simplest formulation of LOP problems is a binary linear programming problem. Several semidefinite
relaxation have been proposed to compute bounds on this challenging combinatorial problem~\cite{HuRe:2013}, the 
basic one obtained from the matrix lifting approach has the following formulation:
\begin{align*}
\max & \,\left\langle C, Z \right\rangle \nonumber \\
 \mbox{ s.t. } & Z \in \sdp, \; \diag(Z) = e,\\
   & y_{ij,jk} - y_{ij,ik} -y_{ik,jk} = -1,  \quad \forall i<j<k,\\ 
\end{align*}
where $ Z =  \begin{bmatrix}
                       1 & y^\top\\
                       y & Y
                      \end{bmatrix}$
is of order $n = {N \choose 2} + 1$.
We have considered LOP instances where the dimension of the set $N$ ranges from $10$ to $100$ 
and the matrix $C$ is randomly generated. 

Again, for these instances,  we have that $\Aa \Aat$ is a diagonal matrix and using~\texttt{DADAL}, especially when dealing 
with large scale instances, leads to an improvement both in terms of number of iterations 
and in terms of computational time (see Table~\ref{tab:LOPinst}).
\begin{table}
\centering      
{\small
  \begin{tabular}{|c||r|r||r||r|r||r|r|} 
   \hline
   \hline  
   \multicolumn{1}{|c||}{}&\multicolumn{2}{c||}{} 
    & \multicolumn{1}{c||}{\texttt{MOSEK}} 
    & \multicolumn{2}{c||}{\texttt{mprw.m}} & 
   \multicolumn{2}{c|}{\texttt{DADAL}}\\
    $|N|$ &  $n$   &  $m$  &   time(s) & iter & time(s) & iter & time(s)\\
  \hline
  \hline
  10 & 46 &   166       &  0.4   & 411 &     \textbf{0.2} &   \textbf{117} &     0.8\\  
  20 & 191 &   1331     &  3.6   & 441 &    \textbf{2.3} &  \textbf{150} &     2.6\\  
  30 & 436 &   4496     &  34.5   & 483 &    \textbf{12.1} &  \textbf{170} &    15.5\\ 
  40 & 781 &   10661    &  294.1   & 542 &    \textbf{60.7} & \textbf{232} &    79.8\\ 
  50 & 1226 &   20826   &  1339.8   & 604 &   237.7 &  \textbf{215} &   \textbf{232.1}\\ 
  60 & 1771 &   35991   &  -  & 636 &   759.8 &  \textbf{252} &   \textbf{748.3} \\
  70 & 2416 &   57156   &  -   & 707 &  \textbf{2049.1} &  \textbf{273} &  2122.7\\ 
  80 & 3161 &   85321   &  -   & 745 &  4788.6 &  \textbf{282} &  \textbf{4395.9}\\ 
  90 & 4006 &   121486  &  -   & 773 &  9589.3 & \textbf{300} &  \textbf{8923.2}\\ 
  100 & 4951 &   166651 &  -   & 821 &  18820.7 & \textbf{323} & \textbf{17944.1}\\
  \hline
  \hline
 \end{tabular}\vspace*{0.1cm}
 \caption{Comparison on linear ordering problems.}
 \label{tab:LOPinst}  
 }
\end{table} 
\section{Conclusions}\label{sec:concl}
We investigate the idea of factorizing the dual variable $Z$ when solving SDPs 
in standard form within augmented Lagrangian approaches.
Our proposal is to use a first order update of the dual variables in order 
to improve the convergence rate of ADAL methods. 
We add this improvement step to the implementation \texttt{mprw.m} and we 
conclude that the approach proposed looks particularly promising for
solving structured SDPs (which is the case for many applications).

From our computational experience we notice that the spectral decomposition needed to perform
the projection is not necessarily the computational bottleneck anymore. In fact,
the matrix multiplications needed in order to update $y$ (see Proposition~\ref{prop:yKeptOpt}) can be the most expensive
operations (e.g. in the case of randomly generated instances).

We also tried to use the factorization of $Z$ in a ``pure'' augmented Lagrangian algorithm. 
However, this turned out to be not competitive
with respect to ADAL methods. This maybe due to the fact that positive semidefiniteness of the primal matrix and complementarity conditions
are not satisfied by construction (as is the case in ADAL methods) and this slows down the convergence.
We want to remark that dealing with the update of the penalty parameter in ADAL methods turned out to be a critical issue. 
Here we propose a unified strategy 
that leads to a satisfactory performance on the instances tested. However, for specific instances, different parameter tuning 
may further improve the performance as demonstrated for the case of computing the $\theta$-number.

Finally, we want to comment about the extension of DADAL to deal with SDPs in general form.
Of course, any inequality constraint may be transformed into an equality constraint via the introduction of a
non-negative slack variable, so that DADAL can be applied to solve any SDP. 
However, avoiding the transformation to SDPs in standard form is in general preferable, in order to preserve favorable constraint 
structures such as sparsity and orthogonality.

In~\cite{Wen2010}, an alternating direction augmented Lagrangian method to deal with SDPs that includes, in particular, positivity
constraints on the elements of the matrix $X$ is presented and tested, even if no convergence analysis is given.
In fact, when considering multi-blocks alternating direction augmented Lagrangian methods, theoretical 
convergence is an issue~\cite{chen2016direct}. 
The investigation on how to properly insert the proposed factorization technique
within converging augmented Lagrangian schemes for SDPs in general form will be the topic of future work.

\section*{Acknowledgement}
The authors would like to thank the anonymous referees for their careful reading of the paper and 
for their constructive comments, which were greatly appreciated.
The third author acknowledges support by the Austrian Science Fund (FWF): I~3199-N31.
\bibliographystyle{elsarticle-num} 
\bibliography{DADAL}

\end{document}